\documentclass{tran-l}
\usepackage{amssymb,amsfonts}
\usepackage[all,arc]{xy}
\usepackage{enumerate}
\usepackage{graphicx}
\usepackage{mathrsfs}
\usepackage{amsmath}
\usepackage{amsfonts}
\usepackage{hyperref}
\usepackage{cleveref}
\usepackage{tikz}
\usepackage{xypic}
\usepackage[utf8]{inputenc}
\usepackage{import}
\usepackage{xifthen}
\usepackage{pdfpages}
\usepackage{transparent}
\usepackage{subfigure}

\newcommand{\R}{\mathbb{R}}
\newcommand{\C}{\mathbb{C}}

\newcommand{\im}{\text{im}}

\newcommand{\Z}{\mathbb{Z}}

\newcommand{\ra}{\rightarrow}

\newcommand{\bb}{\backslash\backslash}
\newcommand{\rdet}[1]{\operatorname{det}^{\operatorname{r}}_{\mathcal N#1}}
\newcommand{\cut}[1]{\stackrel{#1}{\rightsquigarrow}}
\newcommand{\tautwo}{\tau^{(2)}}

\newtheorem{thm}{Theorem}[section]
\newtheorem{cor}[thm]{Corollary}
\newtheorem{prop}[thm]{Proposition}
\newtheorem{lem}[thm]{Lemma}

\newtheorem*{thm1}{Theorem \ref{LeadingCoefficientEqualsRelativeTorsion}}
\newtheorem*{thm2}{Theorem \ref{LeadingCoefficientEqualsGuts}}
\newtheorem*{thm3}{Theorem \ref{InvarianceOfLeadingCoefficientOpenCone}}

\newtheorem{defn}[thm]{Definition}

\newtheorem{rem}[thm]{Remark}

\title[Guts determine the leading coefficients]{Guts determine the leading coefficients \\of $L^2$-Alexander torsions}

\author{Jianru Duan}
\address{Beijing International Center for Mathematical Research, No. 5 Yiheyuan Road,
Haidian District, Beijing 100871, People’s Republic of China}
\email{duanjr@stu.pku.edu.cn}

\begin{document}
\bibliographystyle{amsalpha}
\maketitle

\begin{abstract}
    For 3-manifolds, the leading coefficient of the $L^2$-Alexander torsion is a numerical invariant of a real first cohomology class. We show that the leading coefficient equals the relative $L^2$-torsion of the manifold cut up along a norm-minimizing surface  dual to the cohomology class. Furthermore, the leading coefficient equals the relative $L^2$-torsion of the guts associated to the cohomology class. Finally, we prove that the leading coefficient is constant on any open Thurston cone. The main ingredients are a new criterion for the convergence of Fuglede--Kadison determinants and the work of Agol and Zhang on guts of 3-manifolds.
\end{abstract}

\section{Introduction}
Let $N$ be a connected, orientable, irreducible, compact 3-manifold with empty or toral boundary. For any real cohomology class $\phi\in H^1(N;\R)$, the $L^2$-Alexander torsion introduced by Dubois, Friedl and L\"uck \cite{dubois2015l2} is a function 
\[
    \tautwo(N,\phi):\R_+\longrightarrow[0,+\infty)
\]
well defined up to multiplying by a function $t\mapsto t^r$ for a constant $r\in \R$. For knot complements, a similar invariant called the $L^2$-Alexander invariant was first defined by Li and Zhang \cite{li2006l2, li2006alexander}. The $L^2$-Alexander torsion is similar to the ordinary Alexander polynomial for knot complements in many ways. In particular, the function $t\mapsto\tautwo(N,\phi)(t)$ grows like monomial polynomials when $t$ tends to $0^+$ and $+\infty$, and there is a well-defined leading coefficient of this function, which is a positive real number denoted by $C(N,\phi)$ (see Section \ref{SectionPreliminaries} for details).

Consider an oriented surface $S$ properly embedded in $N$, we denote by $\eta(S)$ or $S\times[-1,1]$ a closed regular neighborhood of $S$, furthermore we write $N\bb S:=N\setminus (S\times (-1,1))$ and $S_\pm:=S\times \{\pm1\}$. In this paper, we will prove the following:

\begin{thm}\label{LeadingCoefficientEqualsRelativeTorsion}
        Let $N$ be a connected, orientable, irreducible, compact 3-manifold with empty or toral boundary. Let $\phi\in H^1(N;\Z)$ and let $S$ be a Thurston norm-minimizing surface dual to $\phi$, then
        \[
        C(N,\phi)=\tautwo(N\backslash\backslash  S,S_-).
        \]
\end{thm}
    The right hand side in the above equality is the relative $L^2$-torsion of the pair $(N\bb S,S_-)$ which has implications in sutured manifold theory \cite{herrmann2023sutured}. Theorem \ref{LeadingCoefficientEqualsRelativeTorsion} was conjectured by Ben Aribi, Friedl and Herrmann in \cite{ben2022leading} motivated by a similar property of the classical Alexander polynomial: the Alexander polynomial $A(t)$ of a knot $K$ has degree no greater than two times the genus of $K$, and if the equality holds, then the leading coefficient of $A(t)$ equals the order of $H_1(X_K\bb S,S_-)$, where $X_K$ is the knot exterior and $S$ is a minimal genus Seifert surface. The trick here is that the Seifert matrix $V$ is invertible and whose determinant is the common value of the leading coefficient and the order of $H_1(X_K\bb S,S_-)$. For the $L^2$-Alexander torsion, a similar phenomenon occurs that there is a matrix $V$ whose Fuglede--Kadison determinant equals the $L^2$-torsion $\tautwo(N\bb S,S_-)$, and $V$ is the limiting matrix of a family $V(t)$ whose Fuglede--Kadison determinant equals $\tautwo(N,\phi)(t)$ when $t<1$. The main difficulty is that the Fuglede--Kadison determinant is in general not continuous. We solve this problem by developing a criterion for the convergence of Fuglede--Kadison determinants (see Theorem \ref{ConvergenceOfFugledeKadison}). 

It is natural to investigate the topological meaning of the leading coefficient $C(N,\phi)$. Liu pointed out in \cite{liu2017degree} that for a primitive class $\phi$, the leading coefficient $C(N,\phi)$ might contain information about the guts if one decomposes $N$ along a maximal collection of non-parallel norm-minimizing surfaces dual to $\phi$. In the terminology of Agol and Zhang \cite{agol2022guts} on the guts of a cohomology class, we can show that the leading coefficient $C(N,\phi)$ is determined by the guts of $\phi$, which is a sutured manifold associated to $N$ and $\phi$, denoted by
\[\Gamma(\phi):=(\Gamma(\phi),R_+ \Gamma(\phi), R_-\Gamma(\phi),\gamma(\Gamma(\phi))).\]
We will prove the following:

    \begin{thm}\label{LeadingCoefficientEqualsGuts}
        Let $N$ be a connected, orientable, irreducible, compact 3-manifold with empty or toral boundary. Then for any primitive $\phi\in H^1(N;\Z)$, we have
        \[
            C(N,\phi)=\tautwo(\Gamma(\phi), R_-\Gamma(\phi)).
        \]
    \end{thm}

Agol and Zhang proved that under mild topological conditions, the guts $\Gamma(\phi)$ is an invariant of the open Thurston cone that contains $\phi$ (see \cite[Theorem 1.2]{agol2022guts}). Combining their result and a convexity argument relying on the author's previous work \cite{duan2022positivity}, we can prove the following:

    \begin{thm}\label{InvarianceOfLeadingCoefficientOpenCone}
        Let $N$ be a connected, orientable, irreducible, compact 3-manifold with empty or toral boundary. Then for each open Thurston cone $\mathcal C$ of $H^1(N;\R)$, the leading coefficient $C(N,\phi)$ is constant for all classes $\phi\in \mathcal C$.
    \end{thm}
    
    As a simple corollary, for fixed $N$ there are only finitely many possible values for the leading coefficient $C(N,\phi)$, and if $\phi$ is in the closure of the unique open Thurston cone containing $\psi$, then $C(N,\phi)\geqslant C(N,\psi)$ (see Corollary \ref{Subordination}).

    The organization of this paper is as follows: in Section \ref{SectionPreliminaries} we review basic notions and collect known results in $L^2$-torsion theory, especially for 3-manifolds. In Section \ref{SectionLeadingCoefficientRelativeTorsion} we give a criterion for the convergence of Fuglede--Kadison determinants and prove Theorem \ref{LeadingCoefficientEqualsRelativeTorsion}. In Section \ref{SectionGutsAndLeadingCoefficient} we review sutured manifold theory and give the proofs of Theorem \ref{LeadingCoefficientEqualsGuts} and Theorem \ref{InvarianceOfLeadingCoefficientOpenCone}.

    \subsection{Acknowledgment}
The author wishes to thank his advisor Yi Liu for guidance and many conversations. The author is indebted to the referee for many detailed suggestions which greatly improves this paper.

\section{Preliminaries}\label{SectionPreliminaries}
For most part of this section, the reader is referred to \cite{luck2002l2} and \cite{ben2022leading} for more details.
\subsection{Hilbert modules and Fuglede--Kadison determinants}

Let $G$ be a group. We form the following Hilbert space 
\[
    l^2(G)=\Big\{\sum_{g\in G}c_g\cdot g\ \Big|\ c_g\in\C,\  \sum_{g\in G}|c_g|^2<\infty\Big\}
\]
with inner product
\[
    \Big\langle\sum_{g\in G}c_g\cdot g,\sum_{g\in G}d_g\cdot g\Big\rangle = \sum_{g\in G} c_g\overline{d_g}.
\]
This Hilbert space has a natural left isometric $G$-action by left multiplication. The ajoint of an element $v=\sum_{g\in G} c_g\cdot g$ in $l^2(G)$ is defined as
\[
    v^*:=\sum_{g\in G} \overline{c_g}\cdot g^{-1}.
\]
Let $\mathcal NG$ be the \textit{group von Neumann algebra of $G$} which consists of all bounded operators on $l^2(G)$ that commutes with the left $G$-action. Any such operator is just the right multiplication by its value on the identity element $e\in G$, so we can identify
$\mathcal NG$ with all elements $x\in l^2(G)$ such that right multiplication with $x$ defines a bounded linear operator on $l^2(G)$ \cite[Definition 2.13]{Kielak2019RFRSGroupsandVirtualFibering}. In particular, $\mathcal NG$ includes the complex group ring $\C G$. The \textit{von Neumann trace} $\operatorname{tr}_{\mathcal NG}:\mathcal NG\ra \C$ is a linear functional sending $x\in \mathcal NG$ to $\langle x,e\rangle\in\C$. Similarly, the von Neumann trace of a square matrix $A\in M_{p,p}(\mathcal NG)$ is the sum of the von Neumann traces of the diagonal entries.

A \textit{finitely generated Hilbert $\mathcal NG$-module} is a Hilbert space with a linear isometric $G$-action, such that it admits a $G$-equivariant linear embedding into a direct sum of finitely many copies of $l^2(G)$. In particular, the Hilbert module $l^2(G)^n$ with the natural isometric left $G$-action is called the \textit{regular Hilbert $\mathcal NG$-module of rank $n$}. A \textit{morphism} between two Hilbert $\mathcal NG$-modules is a bounded $G$-equivariant map.

Let $V=l^2(G)^n$ and $W=l^2(G)^m$ be two regular Hilbert $\mathcal NG$-modules and $f:V\ra W$ be a morphism, then we can identify $f$ with a matrix $A\in M_{n,m}(\mathcal NG)$. On the other hand, for any matrix $A\in M_{n,m}(\mathcal NG)$, the right multiplication defines a bounded $G$-equivariant operator $R_A:l^2(G)^n\ra l^2(G)^m$. The \textit{Fuglede--Kadison determinant} of a matrix $A\in M_{n,m}(\mathcal NG)$ is denoted by $\det_{\mathcal NG} A$, which is a real number in $[0,+\infty)$. We say $A$ is \textit{of determinant class} if $\det_{\mathcal NG} A>0$.

If $A\in M_{n,n}(\mathcal NG)$ is a square matrix, the \textit{regular Fuglede--Kadison determinant} of $A$ is defined to be
\[
    \rdet G(A):=\left\{
    \begin{array}{ll}
    \operatorname{det}_{\mathcal{N}G}(A)     & \text{if $A$ is injective of determinant class,}  \\
    0     &  \text{otherwise.}
    \end{array}\right.
\]
The regular Fuglede--Kadison determinant satisfies many properties similar to the ordinary determinant for square matrices. The following properties follow easily from \cite[Theorem 3.14]{luck2002l2}.

\begin{lem}\label{BasicsOfRegularFugledeKadison}
    Let $G$ be a group and let $n,m$ be positive integers. Then the following holds:
    \begin{itemize}
        \item [(1)] Let $A,B\in M_{n,n}(\mathcal NG)$, then $\rdet G(AB)=\rdet G(A)\cdot \rdet G(B)$.

        \item [(2)] Let $A\in M_{n,n}(\mathcal NG)$, $B\in M_{n,m}(\mathcal NG)$ and $C\in M_{m,m}(\mathcal NG)$, then
        \[\rdet G
        \left(\begin{array}{cc}
            A & B \\
            0 & C 
        \end{array}\right)=\rdet G(A)\cdot \rdet G(C).
        \]

        \item[(3)]  Let $A\in M_{n,n}(\mathcal NG)$, then $\rdet G(A)=\rdet G(A^*)$.

        \item[(4)] Let $i:H\ra G$ be an injective group homomorphism and $A\in M_{n,n}(\mathcal NH)$, then \[\rdet H(A)=\rdet G(i(A)).\]
    \end{itemize}
\end{lem}

The regular Fuglede--Kadison determinant satisfies a weak form of continuity. The following lemma can be found in \cite[Lemma 3.1]{liu2017degree}.

\begin{lem}\label{ContinuityOfFugledeKadison}
If a sequence of $p\times p$ matrices $\{A_n\}$ over $\mathcal{N}G$ converges to $A\in M_{p,p}(\mathcal NG)$ in norm topology, then
    \[
    \limsup_{n\ra \infty} \rdet{G} A_n \leqslant \rdet G A.
    \]
Moreover, if $A$ is a positive operator, then
    \[
    \lim_{\epsilon\ra 0^+} \rdet G (A+\epsilon I)=\rdet G A.
    \]
\end{lem}

Let $\phi\in H^1(G;\R)$ be any real cohomology class and $t>0$ be any positive real number, define the \textit{Alexander twist with respect to $(\phi,t)$} as a ring homomorphism $\kappa(\phi,t):\C G\ra \C G$, with
\[
    \kappa(\phi,t)(g):=t^{\phi(g)}g
\]
for $g\in G$ and then extend the definition to all of the group ring linearly. When $A$ is a matrix over $\C G$, we define $\kappa(\phi,t)A$ by performing the Alexander twist on all entries of $A$. It is easy to verify that $\kappa(\phi,t)$ is a homomorphism of the matrix algebra $M_{m,n}(\C G)$.

We say a function $f:\R_+\ra [0,+\infty)$ is \textit{multiplicatively convex}, if for any $t_1,t_2>0$ and any $\theta\in(0,1)$ we have 
\[f(t_1)^\theta\cdot f(t_2)^{1-\theta}\geqslant f(t_1^\theta\cdot t_2^{1-\theta}).\] When $G$ is a finitely generated, residually finite group, the $L^2$-Alexander twist of a square matrix over $\C G$ enjoys a nice continuity property when taking regular Fuglede--Kadison determinants.
The following theorem is from \cite[Theorem 5.1]{liu2017degree}.

\begin{thm}\label{LiuConvexityOfFugledeKadison}
    Let $G$ be a finitely generated, residually finite group. For any square matrix $A$ over $\C G$ and any cohomology class $\phi\in H^1(G;\R)$, the function
    \[
        t\longmapsto \rdet G (\kappa(\phi,t)A),\quad t\in \R^+
    \]
is either constantly zero or multiplicatively convex (and in particular every where positive). 
\end{thm}

\subsection{$L^2$-Alexander torsions of CW-pairs}\label{SectionTorsionDefinition}
Let $X$ be a finite connected CW-complex with fundamental group $G$ and let $Y$ be a subcomplex of $X$. Let $p:\widehat X\ra X$ be the universal covering of $X$, then $\widehat X$ admits the induced CW-structure coming from $X$. Define $\widehat Y:=p^{-1}(Y)$ to be the subcomplex of $\widehat X$. The natural left $G$-action on $\widehat X$ gives rise to a $\Z G$-module structure for the integral cellular chain complex $C_*(\widehat X,\widehat Y)$. By choosing a lift $\widehat\sigma_i$ for each cell $\sigma_i$ in $X\setminus Y$ to $\widehat X\setminus\widehat Y$, we can identify each module $C_r(\widehat X,\widehat Y)$ with $(\Z G)^{n_r}$, where $n_r$ is the number of $r$-cells in $X\setminus Y$.

For any $\phi\in H^1(X;\R)$ and $t>0$ we can form the chain complex
\[
    C_*^{(2)}(\widehat X,\widehat Y;\phi,t):=l^2(G)\otimes_{\Z G} C_*(\widehat X,\widehat Y)
\]
in which $\Z G$ acts on $l^2(G)$ on the right via the Alexander twist: for any $a\in \Z G$, the action of $a$ on $l^2(G)$ is the right multiplication operator $R_{\kappa(\phi,t)a}$. The connecting homomorphism 
\[
    \partial^{(2)}_r:C_r^{(2)}(\widehat X,\widehat Y;\phi,t)\longrightarrow C_{r-1}^{(2)}(\widehat X,\widehat Y;\phi,t)
\]
is defined to be $\mathrm{id}\otimes_{\Z G} \partial_r$ where $\partial_r:C_r(\widehat X,\widehat Y)\ra C_{r-1}(\widehat X,\widehat Y)$ is the ordinary $\Z G$-linear boundary homomorphism. By our choice of basis, the chain module $C_r^{(2)}(\widehat X,\widehat Y;\phi,t)$ can be identified with $l^2(G)^{n_r}$, the connecting homomorphism $\partial^{(2)}_r$ can be identified with a matrix in $M_{n_r,n_{r-1}}(\C G)$. Let $N$ be the highest dimension of the cells in $X\setminus Y$. We say the chain complex $C_*(\widehat X,\widehat Y;\phi,t)$
\begin{itemize}
    \item [(1)] is \textit{$L^2$-acyclic}, if for all $r$ we have $\ker(\partial^{(2)}_{r})=\overline{\im(\partial^{(2)}_{r+1})}$;

    \item[(2)] is \textit{of determinant class}, if for any $r=1,\ldots,N$, the Fuglede--Kadison determinant $\det_{\mathcal NG}\partial^{(2)}_r$ is nonzero.
\end{itemize}

\begin{defn}
    If the chain complex $C_*(\widehat X,\widehat Y;\phi,t)$ is both $L^2$-acyclic and of determinant class, then we define the $L^2$-Alexander torsion of $(X,Y)$ with respect to $(\phi,t)$ by
\[
    \tautwo(X,Y,\phi)(t):=\prod_{r=1}^N (\operatorname{det}_{\mathcal NG}\partial^{(2)}_r)^{(-1)^r},
\]
if either condition fails we define $\tautwo(X,Y,\phi)(t):=0$. 
\end{defn}
When $\phi=0$ the $L^2$-Alexander torsion is irrelevant of $t$ and we write as $\tautwo(X,Y)$. When $Y=\emptyset$ we write $\tautwo(X,Y,\phi)(t)$ as $\tautwo(X,\phi)(t)$. We call $\tautwo(X)$ the \textit{$L^2$-torsion} of $X$ and $\tautwo(X,Y)$ the \textit{relative $L^2$-torsion} of $(X,Y)$.

Since there are different choices of liftings $\widehat\sigma_i$, the $L^2$-Alexander torsion is only well defined up to multiplying by a monomial $t\mapsto t^r$ for some fixed $r\in \R$. We say two functions $f,g:\R_+\ra [0,+\infty)$ are equivalent, denoted as $f(t)\doteq g(t)$, if there is a constant $r\in \R$ such that $f(t)=t^r\cdot g(t)$ for all $t\in \R_+$.

We remark that the $L^2$-Alexander torsion is invariant under simple homotopy equivalences \cite[Theorem 3.93]{luck2002l2}, in particular, it is invariant under cellular subdivisions.

The $L^2$-Alexander torsions can be defined for smooth manifold pairs as follows: let $M$ be a compact connected smooth manifold, possibly with boundary. Let $N$ be a compact smooth submanifold, possibly with boundary and disconnected. We require that either $N$ is a smooth submanifold of $\partial M$ of co-dimension zero, or the embedding $N\hookrightarrow M$ is proper (i.e. $N\cap \partial M=\partial N$). One can therefore find a smooth triangulation of $M$ such that $N$ is a subcomplex of this triangulation (c.f. \cite[Chapter 10]{Munkres1961ElementaryDiff} and  \cite[Chapter 7]{Verona1984StratifiedMappings}). This triangulation is canonical in the sense that any two smooth triangulations of $M$ admit a common subdivision \cite{Whitehead1940OnC1Complexes}. We use this triangulation to view $(M,N)$ as a simplicial (hence CW) pair $(X,Y)$ and define $\tautwo(M,N,\phi)(t):=\tautwo(X,Y,\phi)(t)$ for any $\phi\in H^1(M;\R)$, $t>0$. 

We define the $L^2$-Alexander torsion of a disconnected manifold pair $(M,N)$ to be the product of the $L^2$-Alexander torsion of the components. Precisely, let $M_1,\ldots,M_n$ be the connected components of $M$ and let $N_i:=M_i\cap N,\ i=1,\ldots,n$. For any $\phi\in H^1(M;\R)$ and $t>0$, define 
\[
    \tautwo(M,N,\phi)(t):=\prod_{i=1}^n \tautwo(M_i,N_i,\phi_i)(t),
\]
where $\phi_i\in H^1(M_i;\R)$ is the restriction of $\phi$ to $M_i$.

The following basic properties are useful. The first one follows from the definition, the others are generalizations of \cite[Theorem 3.93]{luck2002l2} and are stated in \cite[Lemma 2.3]{ben2022leading}.

\begin{lem}\label{BasicOfRelativeTorsion}
Let $(M,N)$ be a compact smooth manifold pair such that $\pi_1(M)$ is residually finite, then:
\begin{itemize}
    \item [(1)] For any $\phi\in H^1(M;\R)$ and any $c\in\R$, we have
    \[\tautwo(M,N,c\phi)(t)=\tautwo(M,N,\phi)(t^c).\]

    \item [(2)] If $M=N\times 
 I$ then for any $s\in[0,1]$ we have
 \[
    \tautwo(M,N\times\{s\})=1.
 \]

    \item[(3)] Let $(M,N)=(X,C)\cup (Y,D)$ where $X$ and $Y$ are submanifolds. Suppose any component of $X\cap Y$ is a submanifold of $\partial X$ and $\partial Y$, any component of $C\cap D$ is a submanifold of $\partial C$ and $\partial D$. Assume further that\linebreak $\tautwo(X\cap Y,C\cap D)>0$ and for any component $Z$ of $X\cap Y$ the induced maps $\pi_1(Z)\ra \pi_1(X)$ and $\pi_1(Z)\ra\pi_1(Y)$ are injective. Then
    \[
        \tautwo(M,N)=\tautwo(X,C)\cdot\tautwo(Y,D)\cdot\tautwo(X\cap Y,C\cap D)^{-1}.
    \]\end{itemize}
    
\end{lem}

\subsection{$L^2$-Alexander torsion of 3-manifolds}
Let $N$ be a compact oriented 3-manifold, for any integral homology class $z\in H_2(N,\partial N)$ we define the \textit{Thurston norm} of $z$ to be
\[
    x_N(z):=\min\left\{\chi_-(\Sigma)\mid \text{$\Sigma$ properly embedded surface with $[\Sigma,\partial \Sigma]=z$}\right\}
\]
where given a compact oriented surface $\Sigma$ with components $\Sigma_1,\ldots,\Sigma_k$ we define its complexity as
\[
    \chi_-(\Sigma):=\sum_{i=1}^k \max\{0,-\chi(\Sigma_i)\}.
\]
Let $\Sigma$ be an oriented properly embedded surface representing $z\in H_2(N,\partial N)$. We say $\Sigma$ is \textit{norm-minimizing} if $\Sigma$ has no null-homologous collection of components, and $\chi_-(\Sigma)=x_N(z)$. We say $\Sigma$ is \textit{taut} if $\Sigma$ is incompressible, has no null-homologous collection of components, and $\Sigma$ has minimal complexity among all properly embedded oriented surfaces representing the homology class $[\Sigma,\partial \Sigma]\in H_2(N,\eta(\partial \Sigma))$. Here $\eta(\partial \Sigma)$ is the tubular neighborhood of $\partial \Sigma$ in $\partial N$. The definitions here follow from \cite{agol2022guts}. When $N$ is irreducible and is not homeomorphic to $S^1\times D^2$, the definitions coincide with the ones in \cite{ben2022leading}.

Thurston proved in \cite{thurston1986norm} that $x_N$ can be extended to a pseudo-norm of the real vector space $H_2(N,\partial N;\R)$ and whose unit ball $B_x(N)$ is a finite convex (possibly non-compact) polyhedron. An \textit{open Thurston cone} is either an open cone formed by the origin and a face of $B_x(N)$, or a maximal connected component of $H_2(N,\partial N;\R)-\{0\}$ on which $x_N$ vanishes. It turns out that $H_2(N,\partial N;\R)-\{0\}$ is the disjoint union of all open Thurston cones. One should note that open Thurston faces or cones are usually not open sets; they are cells in $H_2(N,\partial N;\R)$ whose dimensions can be various.

We often identify $H_2(N,\partial N)$ with $H^1(N)$ by Poincar\'e duality and all the above structures can also be defined on the cohomology group $H^1(N)$ (with integral or real coefficients). A cohomology class $\phi\in H^1(N;\R)$ is called \textit{fibered} if the class is represented by a nowhere vanishing differential 1-form on $N$. An integral class $\phi$ is fibered if and only if $N$ admits a fibration over the circle $p:N\ra S^1$ such that $\phi$ is the pull back class of the fundamental class $[S^1]\in H^1(S^1;\Z)$. Thurston proved that the set of fibered classes is some union of top dimensional open Thurston cones \cite{thurston1986norm}.

When the pseudo-norm $x_N$ becomes a norm then we say that $N$ has \textit{non-degenerate Thurston norm}.

 We remark that the fundamental group of a 3-manifold is residually finite (see \cite{hempel1987residual} and \cite{aschenbrenner20153}), so Lemma \ref{BasicOfRelativeTorsion} applies to 3-manifold pairs. The following theorem collects known facts relating to the $L^2$-Alexander torsion of 3-manifolds.
 
\begin{thm}\label{KnownPropertiesOfL2Alexander}
    Let $N$ be an orientable irreducible compact 3-manifold with empty or toral boundary. Let $\phi\in H^1(N;\R)$ be a real cohomology class.

    \begin{itemize}

    \item [(1)]Let $\operatorname{Vol}(N)$ be the sum of the volumes of the hyperbolic pieces in the JSJ-decomposition of $N$ (c.f. \cite[Theorem 1.7.6]{aschenbrenner20153}), then we have
    \[
        \tautwo(N,\phi)(1)=e^{\frac{\operatorname{Vol}(N)}{6\pi}}.
    \]

    \item[(2)]  If $p:M\ra N$ is a covering of finite degree $d$, then
    \[
        \tautwo(M,p^*\phi)(t)\doteq\tautwo(N,\phi)(t)^d
    \]
    where $p^*\phi$ is the pull back class of $\phi$.

    \item[(3)] If $N$ is obtained from a (possibly disconnected) 3-manifold $M$ via gluing together pairs of incompressible toral boundaries of $M$, then
    \[
        \tautwo(N,\phi)(t)\doteq \tautwo(M,\phi|_M)(t)
    \]
    where $\phi|_M$ is the restriction of $\phi$ to $M$.

    \item [(4)] If $N$ is a graph manifold, then $\tautwo(N,\phi)(t)\doteq \max\{1,t\}^{x_N(\phi)}$.

    \item[(5)] If $\phi$ is a fibered class, then 
    \[
        \tautwo(N,\phi)(t)\doteq
    \begin{cases}
        1 & \text{if\ }t<\frac1{h(\phi)},\\
     t^{x_N(\phi)}& \text{if\ }t>h(\phi)
    \end{cases}
    \]
    where $h(\phi)$ is the entropy of the monodromy.

    \item[(6)] $\tautwo(N,\phi)(t)$ is strictly positive for all $t>0$.

    \item[(7)] $\tautwo(N,\phi)(t)$ is a continuous function of $t$. Moreover, there exists constants $b_1,\ b_2\in\R$ and $C(N,\phi)\geqslant 1$ such that 
    \[
        \lim_{t\ra 0^+}\frac{\tautwo(N,\phi)(t)}{t^{b_1}}= \lim_{t\ra +\infty}\frac{\tautwo(N,\phi)(t)}{t^{b_2}}=C(N,\phi)
    \]
    with $b_2-b_1=x_N(\phi)$.

    \end{itemize}
\end{thm}
The first statement follows from the definition and \cite{luck19992}. The statements (2)--(4) follow from Lemma 5.3, Theorem 5.5, Theorem 1.2 of \cite{dubois2015l2} respectively. The statement of (5) is proved for rational classes in \cite[Theorem 1.3]{dubois2015l2}, and is generalized to real classes in \cite[Theorem 5.10]{duan2022positivity}. The statement (6) is independently proved in \cite[Theorem 1.2]{liu2017degree} and \cite[Theorem 7.7]{luck2018twisting}. The statement (7) is from \cite[Theorem 1.2]{liu2017degree}; when $\phi$ is a rational class, the existence of $b_1, b_2\in \R$ satisfying $b_2-b_1=x_N(\phi)$ is independently proved by \cite[Theorem 0.1]{friedl2019l2}.

 We call $C(N,\phi)$ the \textit{leading coefficient} of the $L^2$-Alexander torsion of $N$ with respect to $\phi$. We briefly call $C(N,\phi)$ the leading coefficient of $\phi$ when there is no confusion caused. Here are some known facts about the leading coefficient.

 \begin{thm}\label{PropertiesOfLeadingCoefficient}
     Let $N$ be a connected, orientable, irreducible, compact 3-manifold with empty or toral boundary.  Let $\phi\in H^1(N;\R)$ be a real cohomology class.

     \begin{itemize}
         \item [(1)] The leading coefficient is dilatation invariant, i.e. $C(N,c\phi)=C(N,\phi)$ for any nonzero real number $c$. Let $\operatorname{Vol}(N)$ be the sum of the volumes of the hyperbolic pieces in the JSJ-decomposition of $N$, then
         \[C(N,0)=e^{\frac{\operatorname{Vol}(N)}{6\pi}}.\]

         \item [(2)] If $p:M\ra N$ is a covering of finite degree $d$, then
    \[
        C(M,p^*\phi)=C(N,\phi)^d
    \]
    where $p^*\phi$ is the pull back class of $\phi$.

        \item[(3)] If $N$ is obtained from a (possibly disconnected) 3-manifold $M$ via gluing together pairs of incompressible toral boundaries of $M$, then
    \[
        C(N,\phi)=C(M,\phi|_M)
    \]
    where $\phi|_M$ is the restriction of $\phi$ to $M$.

        \item[(4)]  If $N$ is a graph manifold, then $C(N,\phi)=1.$
         
        \item[(5)] If $\phi$ is a fibered class, then $C(N,\phi)=1.$

        \item[(6)] The mapping $\phi\mapsto C(N,\phi)$ is upper semi-continuous with respect to\linebreak $\phi\in H^1(N;\R)$.

        \item[(7)] If $\phi\in H^1(N;\Z)$ is a primitive class and $\Sigma$ is a properly embedded norm-minimizing surface dual to $\phi$, then
    \[
        \prod_{\substack{M\subset N \text{ is a hyperbolic piece}\\\text{with } \phi|_M=0}
            }e^{\frac{\operatorname{Vol}(M)}{6\pi}}\leqslant
            C(N,\phi)\leqslant \tautwo(N\bb \Sigma,\Sigma_-)
    \]
    where $N\bb \Sigma:=N\setminus(\Sigma\times (-1,1))$ is the manifold cut open along $\Sigma$, and $\Sigma_\pm=\Sigma\times\{\pm1\}$.
     \end{itemize}
     
 \end{thm}
The statements (1)--(5) are immediate consequences of Lemma \ref{BasicOfRelativeTorsion} and Theorem \ref{KnownPropertiesOfL2Alexander}. Statement (6) is proved in \cite[Theorem 1.2]{liu2017degree}, and statement (7) is proved in \cite[Theorem 1.6]{ben2022leading}.

\section{The leading coefficient and the relative $L^2$-torsion}\label{SectionLeadingCoefficientRelativeTorsion}
In this section we prove Theorem \ref{LeadingCoefficientEqualsRelativeTorsion}, relating the leading coefficient to the relative $L^2$-torsion. The main tool is a convergence theorem for Fuglede--Kadison determinants.

\begin{defn}
    Let $G$ be a group and $\phi\in H^1(G;\R)$ be a nonzero cohomology class. Suppose $A$ is a matrix over $\C G$. We say that $A$ is $\phi$-positive if every group element $g\in G$ appearing in $A$ satisfies $\phi(g)>0$.
\end{defn}
We remark that if a square matrix $A$ is $\phi$-positive then the von Neumann trace $\operatorname{tr}_G A=0$, since the identity element does not appear in $A$.

\begin{thm}\label{InequalityOfFugledeKadison}
    Let $G$ be a finitely generated, residually finite group and $\phi\in H^1(G;\R)$ be a nonzero cohomology class. Let $P,Q$ be two square matrices over $\C G$ with the same size. Suppose $P$ is over $\C[\ker \phi]$ and $Q$ is $\phi$-positive. Then \[\rdet{G}(P+Q) \geqslant \rdet{G}P.\]
\end{thm}

Before we prove the theorem, we need the following lemma from \cite[Theorem 1.10 (e)]{mathai1997determinant}.

\begin{lem}\label{CarayFarberMathai}
    Let $G$ be a group, let $t\in \R^+$ and let
    \[
        A:[0,t]\longrightarrow \operatorname{GL}(n,\mathcal NG),\quad s\longmapsto A(s)
    \]
    be a continuous piecewise smooth map, then
    \[\operatorname{det}_{\mathcal{N}(G)}^{\mathrm{r}}(A(t))=\operatorname{det}_{\mathcal{N}(G)}^{\mathrm{r}}(A(0)) \cdot \exp \left(\int_{0}^{t} \operatorname{Re} \operatorname{tr}_G\left(A(s)^{-1} \cdot A'(s)\right) d s\right).\]
\end{lem}

\begin{proof}[{Proof of Theorem \ref{InequalityOfFugledeKadison}}]
    If $\rdet GP=0$, then nothing needs proving. We suppose that $\rdet GP>0$. We can further assume without loss of generality that $P$ is positive: this is because the desired inequality is equivalent to $\rdet G(P^*P+P^*Q)\geqslant \rdet G(P^*P)$ in which $P^*Q$ is $\phi$-positive and $P^*P$ is positive with entries in $\C[\ker \phi]$.
    
    For any $t>0$, define $Q(t):=\kappa(\phi,t)Q$ to be the Alexander twist of $Q$ with respect to $(\phi,t)$. Since $Q$ is $\phi$-positive, we know that the operator norm $\|Q(t)\|\ra 0$ as $t\ra 0^+$. Now for any fixed $\epsilon>0$, there exists $t_0>0$ such that for any $s< t_0$, we have $\|Q(s)\|<\|P+\epsilon I\|$, hence the matrix $P+\epsilon I+Q(s)$ is invertible and the following power series converges in norm topology:
    \[
        (P+\epsilon I+Q(s))^{-1}\cdot Q'(s)=(P+\epsilon I)^{-1}\cdot \sum_{n=0}^\infty  \big(Q(s)(P+\epsilon I)^{-1}\big)^n\cdot Q'(s)
    \]
Note that $P+\epsilon I$ is invertible over the von Neumann algebra $\mathcal N(\ker\phi)$, its inverse $(P+\epsilon I)^{-1}$ is also in $\mathcal N(\ker\phi)$.  Since $Q(s)$ and $Q'(s)$ are $\phi$-positive, any term of the power series has zero von Neumann trace. Combining the fact that the von Neumann trace is continuous with respect to weak topology (and hence to norm topology, see \cite[Theorem 2.35]{kammeyer2019introduction}), we have
    \[      \operatorname{tr}_G\Big(\big(P+\epsilon I+Q(s)\big)^{-1} \cdot Q'(s)\Big)=0
\]
for all $s<t_0$. So by Lemma \ref{CarayFarberMathai} the following identities

    \[ 
    \begin{aligned}
          &\rdet G(P+\epsilon I+Q(t))\\&\quad=
          \rdet G(P+\epsilon I)\cdot\exp\left(\int_{0}^{t} \operatorname{Re} \operatorname{tr}_G\Big(\big(P+\epsilon I+Q(s)\big)^{-1} \cdot Q'(s)\Big) d s\right)\\
          &\quad=\rdet G(P+\epsilon I)
    \end{aligned}  
    \]
    hold for any $t< t_0$.
    
    Observe that 
    \[
        P+\epsilon I+Q(t)=\kappa(\phi,t)(P+\epsilon I+Q),
    \]
    Theorem \ref{LiuConvexityOfFugledeKadison} applies that the function
    \[
        t\longmapsto \rdet G(P+\epsilon I+Q(t)),\quad t\in \R_+
    \]
    is either constantly zero or multiplicatively convex. This function is already constant for $t<t_0$, so by convexity it is non-decreasing over $\R_+$. In conclusion, for any $\epsilon>0$ and $t>0$,
    \[
        \rdet G(P+\epsilon I+Q(t))\geqslant\rdet G(P+\epsilon I).
    \]
     So we have the following:
    \[\aligned
        \rdet G(P+Q(t))&\geqslant \limsup_{\epsilon\ra 0^+} \rdet G(P+\epsilon I+Q(t))\\&\geqslant\limsup_{\epsilon\ra 0^+} \rdet G(P+\epsilon I)\\&=\rdet GP
        \endaligned
    \]
    in which the first inequality and the last equality are consequences of Lemma \ref{ContinuityOfFugledeKadison}. Set $t=1$, we obtain the desired inequality
    \[
        \rdet G(P+Q)\geqslant \rdet GP.
    \]
    \end{proof}

     The following theorem offers a criterion for the convergence of Fuglede--Kadison determinants.

     \begin{thm}
        Let $G$ be a finitely generated, residually finite group. Let $P$ be a square matrix over $\C G$, let $ Q(t)$ be a family of square matrices over $\C G$ indexed by $t>0$, each with the same size as $P$. We don't require that the family $Q(t)$ is continuous. Suppose that 
        \begin{itemize}
            \item[(1)] the operator norm $\|Q(t)\|\ra 0$ as $t\ra 0^+$, and

            \item[(2)] there exists a family of nonzero cohomology classes $\phi_t\in H^1(N;\R)$ indexed by $t>0$, such that for $t$ sufficiently small, $P$ is over $\C[\ker\phi_t]$ and $Q(t)$ is $\phi_t$-positive.
        \end{itemize}
         Then 
         \[
            \lim_{t\ra 0^+} \rdet G(P+Q(t))=\rdet GP.
         \]
    \end{thm}
    \begin{proof}
        The assumption $(2)$ allows us to apply Theorem \ref{InequalityOfFugledeKadison} to conclude that\linebreak $\rdet G(P+Q(t))\geqslant \rdet GP$ for any sufficiently small $t>0$, hence \[\liminf_{t\ra 0^+} \rdet G(P+Q(t))\geqslant \rdet GP.\] On the other hand we have \[\limsup_{t\ra 0^+} \rdet G(P+Q(t))\leqslant \rdet GP\]
        by Lemma \ref{ContinuityOfFugledeKadison}. These two inequalities implies the desired equality.
    \end{proof}

    The following is an immediate corollary.
     \begin{cor}\label{ConvergenceOfFugledeKadison}
        Let $G$ be a finitely generated, residually finite group and $\phi\in H^1(G;\R)$ be a nonzero cohomology class. Let $P,\ Q$ be two square matrices over $\C G$ with the same size. Suppose $P$ is over $\C[\ker \phi]$ and $Q$ is $\phi$-positive. Then 
         \[
            \lim_{t\ra 0^+} \rdet G(P+\kappa(\phi,t)Q)=\rdet GP.
         \]
    \end{cor}

    We are ready to prove Theorem \ref{LeadingCoefficientEqualsRelativeTorsion}. It is implicit in the proof of \cite[Theorem 1.6]{ben2022leading} that Corollary \ref{ConvergenceOfFugledeKadison} implies Theorem \ref{LeadingCoefficientEqualsRelativeTorsion}. We include a sketched proof here for the reader's convenience. It is notable that we use Corollary \ref{ConvergenceOfFugledeKadison} to directly identify the leading coefficient without having to analyze the asymptotic degree when $t\ra 0^+$.

    \begin{thm1}
        Let $N$ be a connected, orientable, irreducible, compact 3-manifold with empty or toral boundary. Let $\phi\in H^1(N;\Z)$ and let $S$ be a Thurston norm-minimizing surface dual to $\phi$, then $C(N,\phi)=\tautwo(N\backslash\backslash  S,S_-)$.
    \end{thm1}
    
    \begin{proof}
        
        By the doubling trick in Step 3 of the proof of \cite[Theorem 1.6]{ben2022leading}, we only need to prove the theorem for closed $N$.
        
        By \cite[proposition 3.3]{ben2022leading}, there exists an oriented surface $\Sigma$ with connected components $\Sigma_1,\ldots,\Sigma_l$ and positive integers $w_1,\ldots,w_l$ such that the following hold.
        \begin{itemize}
            \item  The class $\sum_{i=1}^l w_i[\Sigma_i]\in H_2(N;\Z)$ is dual to $\phi$.

        \item  We have $\sum_{i=1}^l -w_i\chi(\Sigma_i)=x_N(\phi)$.

        \item  $N\bb\Sigma$ is connected.

        \item  $\tautwo(N\bb S,S_-)=\tautwo(N\bb \Sigma,\Sigma_-)$.
        \end{itemize}

         Define $M:=N\bb \Sigma$. It is clear that we only need to prove $C(N,\phi)=\tautwo(M,\Sigma_-)$. Construct a CW-structure for $N$ with the following properties as in \cite{ben2022leading}:

        \begin{itemize}
            \item 
                $M$ and $\Sigma\times [-1,1]$ are subcomplexes.
            \item 
                $\Sigma$ has exactly one $0$-cell $p_i$ in each component $\Sigma_i$ and the CW-structure on $\Sigma\times[-1,1]$ is a product CW-structure.
                
        \item  $M$ has precisely one $3$-cell $\beta$.

        \item  There is exactly one $0$-cell $q$ in the interior $M\setminus(\Sigma\times\{\pm 1\})$.

        \item  For each $i=1,\ldots,l$ there are $1$-cells $\nu_i^\pm$ connecting $q$ and $p_i^\pm=p_i\times\{\pm1\}$ lying entirely in $M$.
        \end{itemize}
        We remark that this CW-structure is obtained from an arbitrary smooth triangulation by doing elementary collapsing and expansions (which do not change the $L^2$-Alexander torsion). So this CW-structure can be used to calculate the $L^2$-Alexander torsion.

        We denote by $\gamma_i$ the element in $\pi_1(N,q)$ induced by the path concatenating $\nu^-_i, \nu_i^+$ and $p_i\times I$. By our construction of $\Sigma$ we know that $\phi(\gamma_i)=w_i$ and $\pi_1(M,q)$ is in the kernel of $\phi$. By lifting each cell of $N$ to the universal covering $\widehat N$ carefully, we obtain an explicit formula for the $L^2$-Alexander torsion $\tautwo(N,\phi)(t)$. In fact, as in the proof of \cite[Theorem 1.6]{ben2022leading} there exists matrices $B,C,D^+,D^-,\partial$ and $D(t)$ of appropriate size, such that:
        \begin{itemize}
            
        \item  $B,C,D^+,D^-,\partial$ are matrices over the subring $\Z[\pi_1(M,q)]\subset \Z[\ker \phi]$.

        \item  $D(t)$ is a diagonal matrix, with each entry of the form $-t^{w_i}\gamma_i$ for some $i\in\{1,\ldots,l\}$.

        \item  The equality holds:
        \[
            \tautwo(M,\Sigma_-)=\rdet{\pi_1(N,q)}
        \left(\begin{array}{ccc}
            B & C & D^+ \\
            0 & 0 &  \partial 
        \end{array}\right).
        \]

        \item  The function $f:\R^+\ra \R^+$:
        \[f(t)=\rdet{\pi_1(N,q)}
        \left(\begin{array}{cccc}
            B & C & D^+ & D^- \\
            0 & 0 &  \partial & 0\\
            0 & 0 & D(t) & Id
        \end{array}\right)\cdot \max\{1,t\}^{-2}
        \]
        is a representative for the $L^2$-Alexander torsion $\tautwo(N,\phi)(t)$.
       
        \end{itemize}
        We remark that $D(t)=\kappa(\phi,t)D(1)$ and $D(1)$ is $\phi$-positive. So by Corollary \ref{ConvergenceOfFugledeKadison}, we have
        \[
            \lim_{t\ra 0^+} f(t)=\rdet{\pi_1(N,q)}
        \left(\begin{array}{cccc}
            B & C & D^+ & D^- \\
            0 & 0 &  \partial & 0\\
            0 & 0 & 0 & Id
        \end{array}\right)=\tautwo(M,\Sigma_-).
        \]
        Since the limit is a positive real number, we conclude that this limit is exactly the leading coefficient. Hence $C(N,\phi)=\tautwo(M,\Sigma_-)$.
    \end{proof}

\section{The guts and the leading coefficient}\label{SectionGutsAndLeadingCoefficient}
   In this section, our terminologies and conventions of the sutured manifold theory follow from \cite{agol2022guts}. We will use Theorem \ref{LeadingCoefficientEqualsRelativeTorsion} and work of Agol and Zhang \cite{agol2022guts} to relate the leading coefficient to the guts of the cohomology class.

    \subsection{Preliminaries to sutured manifolds}\label{PreliminarySuturedManifold}

    \begin{defn}
        A sutured manifold $(M,\gamma)$ is a compact oriented 3-manifold $M$ with boundary together with a set $\gamma\subset \partial M$, consisting of disjoint annuli $A(\gamma)$ and tori $T(\gamma)$. A component of $A(\gamma)$ is called a sutured annulus, and a component of $T(\gamma)$ is called a sutured torus. Any sutured annulus contains a {suture}, which is an oriented simple closed curve homotopic to the core curve. The union of all the sutures is called $s(\gamma)$.

        Furthermore, every component of $R(\gamma):=\partial M\backslash \operatorname{int}(\gamma)$ is oriented. Define $R_+$ (or $R_-$) to be those components of $R(\gamma)$ whose normal vectors point out of (or into) $M$. We require that the orientation of the sutures must be coherent with the boundary orientation of $R(\gamma)$.

        A sutured manifold $(M,\gamma)$ is {taut} if $M$ is irreducible and $R_+$ and $R_-$ are taut surfaces.
    \end{defn}

    \begin{defn}
        Let $(M,\gamma)$ be a sutured manifold. Suppose $S$ is a properly embedded oriented surface such that for every component $\lambda$ of $S\cap \gamma$ one of the following holds:

        \begin{itemize}
            \item [(1)] 
                $\lambda$ is a properly embedded non-separating arc in a sutured annulus,
            \item [(2)]
                $\lambda$ is a simple closed curve in a sutured annulus $A$ representing the same homology class of the suture $s(\gamma)\cap A$,

            \item [(3)]
                $\lambda$ is a homotopically nontrivial simple closed curve in a sutured torus $T$. If $\delta$ is another such component of $S\cap\gamma$ then $\gamma$ and $\delta$ represent the same homology class in $H_1(T;\Z)$.
        \end{itemize}
    Furthermore, suppose no component of $\partial S$ bounds a disk in $R(\gamma)$ and no component of $S$ is a disk $D$ with $\partial D\subset R(\gamma)$. Then we call $S$ a decomposition surface for $(M,\gamma)$. The boundary of the tubular neighborhood $\eta(S):= S\times[-1,1]$ is denoted $S_\pm$ with the induced orientation. We then define a suture structure on $M':=\overline{M\setminus \eta(S)}$:
    \[\aligned &\gamma'=(\gamma\cap M')\cup {\eta(S_+\cap R_-)}\cup {\eta(S_-\cap R_+)},\\&
        R_+'=\big((R_+\cap M')\cup S_+\big)\setminus \operatorname{int}(\gamma'),\\&
        R_-'=\big((R_-\cap M')\cup S_-\big)\setminus \operatorname{int}(\gamma'). \endaligned
    \]
    We say that $(M,R_+,R_-,\gamma)\cut S(M',R_+',R_-',\gamma')$ is a sutured manifold decomposition, and write as $(M,\gamma)\cut S (M',\gamma')$ for short.
    \end{defn}

   \begin{defn}
       Let $(M, \gamma)$ be a sutured manifold. A product annulus is a properly embedded annulus $A$ in $M$ which does not cobound a solid cylinder in $M$ and such that one boundary component of $A$ lies on $R_-$ and the other one lies on $R_+$. A product disk is a properly embedded disk $D$ in $M$ such that $|D\cap s(\gamma)|=2$.
   \end{defn}

    \begin{defn}
        Suppose $z$ is a primitive element in $H_2(M, \partial M;\Z)$. Take a maximal collection of disjoint, non-parallel, properly norm-minimizing decomposition surfaces $\Sigma_1,\ldots,\Sigma_k$ in $M$ such that $[\Sigma_i]=z\in H_2(M, \partial M;\Z)$ for $i=1,\ldots,k$. We call the union of $\Sigma_1,\ldots,\Sigma_k$ a facet surface for $z$ and denote $\Sigma_1\cup \cdots \cup \Sigma_k$ as $F(z)$.
    \end{defn}
    
        Let $(M,\gamma)$ be a taut sutured manifold, and let $S$ be a maximal set of pairwise disjoint non-parallel product annuli and product disks in $M$. Let $(\overline M,\overline \gamma)$ be the sutured manifold obtained from $(M,\gamma)$ by decomposing along S and throwing away product sutured manifold components. Then $(\overline M,\overline \gamma)$ is called the \textit{guts} of $(M,\gamma)$. 

        \begin{rem}
            The definition of guts relies on the Kneser--Haken finiteness: any collection of disjoint, non-parallel incompressible surfaces in $M$ with first Betti number at most $b$ must have less than $c$ members, where $c=c(M,b)$ is a constant depending on the $3$-manifold $M$ and $b$. We refer to \cite[Theorem 1]{Freedman1998KneserHakenFiniteness} for a precise statement.
        \end{rem}

        Let $N$ be a connected, orientable, irreducible, compact 3-manifold with empty or toral boundary, we consider $(N,\partial N)$ as a taut sutured manifold with $\gamma(N)=\partial N$.

        \begin{defn}
 
        If $F$ is a facet surface for a primitive element $z$ in $H_2(N,\partial N;\Z)$, then the guts of the sutured manifold $N\bb F$ is called the guts of $z$, and is denoted as $\Gamma(z)$.
        \end{defn}

        The definition of $\Gamma(z)$ depends on the choice of facet surfaces. Two guts of $z$ are called \textit{equivalent} if there is an isotopy of $N$ which restricts to a homeomorphism between the two guts as sutured manifolds. When $N$ has non-degenerate Thurston norm, it is proved in \cite[Theorem 1.1]{agol2022guts} that the guts $\Gamma(z)$ do not depend on the choice of facet surface $F(z)$ up to the equivalence of guts. More surprisingly, the following Theorem \ref{AgolZhangTheorem} of Agol and Zhang \cite[Theorem 1.2]{agol2022guts} shows that, under a mild condition, the guts of $z$ is an invariant of the open Thurston cone containing $z$. Let $P$ be a boundary component of $N$ and let $\partial_P:H_2(N,\partial N;\R)\ra H_1(P;\R)$ be the following composition of mappings: $$H_2(N,\partial N;\R)\xrightarrow{\partial} H_1(\partial N;\R)\xrightarrow{\pi}H_1(P;\R)$$ where $\partial$ is from the homology long exact sequence of $(N,\partial N)$ and $\pi$ is the projection map. We say two cohomology classes $u,v\in H^1(N;\R)$ are \textit{in opposite orientations on $P$} if 
            $\partial_P u\not=0$, $\partial_P v\not=0$, and $\partial_P u=c\cdot \partial_P v$ for some $c<0$.

        \begin{thm}\label{AgolZhangTheorem}
            Let $N$ be an orientable, irreducible, compact 3-manifold with empty or toral boundary and non-degenerate Thurston norm. Let $y,z$ be two primitive elements in an open Thurston cone. If there is an open segment $(v,w)$ containing $y,z$ in the open Thurston cone such that the restrictions of $v$ and $w$ on each boundary component are not in opposite orientations, then the guts $\Gamma(y)$ is equivalent to $\Gamma(z)$.
        \end{thm}

        We deduce from Theorem \ref{AgolZhangTheorem} the following corollary, which is exactly what we need later for the proof of Theorem \ref{InvarianceOfLeadingCoefficientOpenCone}.
        \begin{cor}\label{FinitelyManyGutsInOpenThurstonCone}
        Let $N$ be an orientable, irreducible, compact 3-manifold with empty or toral boundary and non-degenerate Thurston norm. Then there are only finitely many equivalent classes of guts $\Gamma(z)$ for primitive $z$ in an open Thurston cone $\mathcal C$.
    \end{cor}
    \begin{proof}
        Suppose on the contrary that there is an infinite sequence of primitive classes 
    $\{z_i\in \mathcal C: i\geqslant 1\}$ with pairwise non-equivalent guts $\Gamma(z_i)$. For any boundary component $P$, there are two possibilities:
    \begin{itemize}
        \item[(1)] if $\{\R\cdot\partial_P(z_i) :i\geqslant 0\}$ contains infinitely many distinct real lines in the vector space $H_1(P;\R)$, then there is an infinite subsequence $z_i'$ such that $\partial_Pz_i'$ and $\partial_Pz_j'$ are linearly independent for all $i\not=j$;
        
        \item[(2)] otherwise, $\{ \R\cdot\partial_P(z_i):i\geqslant 0\}$ contains only finitely many real lines in $H_1(P;\R)$, then there is a finite set  $\Omega:=\{\phi_1,\ldots,\phi_k\}$ of nonzero primitive classes in $H_1(P;\Z)$ such that for all $i$ either $\partial_Pz_i=0$ or $\partial_Pz_i$ is a positive multiple of an element in $\Omega$.
    \end{itemize}
    In either case, we can find an infinite subsequence $\{z_i'\in \mathcal C:i\geqslant1\}$ such that one of the following holds: 
\begin{itemize}
    \item[(i)]  $\partial_Pz_i'$ is linearly independent with $\partial_Pz_j'$ for all $i\not= j$. 
    \item[(ii)]  $\partial_Pz_i'=0$ for all $i\geqslant 1$. 
    \item[(iii)] $\partial_Pz_i'$ is a positive multiple of a fixed nonzero class for all $i\geqslant1$.
\end{itemize}
    The key property we want for the sequence is: for any $i\not=j$, there is an open segment $(u,v)$ in $\mathcal C$ containing $z_i'$ and $z_j'$ with $\partial_Pu$ and $\partial_Pv$ not in the opposite orientations. In fact, for case (i) and (iii), any open segment $(u,v)$ containing $z_i'$ and $z_j'$ with $\{u,v\}$ sufficiently close to $\{z_i',z_j'\}$ suffices; for case (ii), any open segment $(u,v)$ containing $z_i',z_j'$ will satisfy $\partial_Pu=\partial_Pv=0$.
    
    The subsequence $\{z_i'\in \mathcal C:i\geqslant1\}$ enjoys the above key property for the boundary component $P$. Passing to further subsequences, we assume that this property holds for every boundary component of $N$, and the new sequence is still denoted as $\{z_i'\in \mathcal C:i\geqslant1\}$.
    Then by construction there is an open segment $(u,v)$ in $\mathcal C$ containing $z_1'$ and $z_2'$ with $\partial_Pz_1'$ and $\partial_Pz_2'$ not in the opposite orientations on any boundary component $P$ of $N$. By Theorem \ref{AgolZhangTheorem}, $\Gamma(z_1')$ is equivalent to $\Gamma(z_2')$, which is a contradiction.
    \end{proof} 
    \begin{rem}
        There is a similar result in \cite[Theorem 2.3]{Boileau2014Finitenessof3-manifolds} where the finiteness of ``patterned guts" of a closed 3-manifold is established using classical normal surface techniques. 
    \end{rem}

    \subsection{The leading coefficient is determined by the guts}

We are going to prove Theorem \ref{LeadingCoefficientEqualsGuts}. The following lemma tells us that decomposition along product disks and incompressible product annuli do not alter the relative $L^2$-torsion.
    \begin{lem}\label{CuttingAlongProductAnnuli}
        Let $(M, R_+, R_-, \gamma)$ be a taut sutured manifold. Let $(C, \partial C)\subset(M, R_+\cup R_-)$ be a disjoint collection of product disks and product annuli. Let $(M,\gamma) \stackrel{C}{\rightsquigarrow} (M',\gamma')$ be the sutured manifold decomposition. Then we have
        \[
            \tautwo(M,R_-)=\tautwo(M',R_-').
        \]
    \end{lem}

     \begin{proof}
        It is a result of Gabai \cite[Lemma 0.4]{gabai1987foliations} that sutured decomposition along product annuli and product disks preserves tautness. So it suffices to prove the equality for $C$ a single product annulus or product disk.
     
        Suppose that $C$ is a product annulus. The embedding of $C$ in $M$ must be $\pi_1$-injective, for otherwise each component of $\partial C$ is contractible in $M$, then one of the component of $\partial C$ must bound a disk in $R_+$ and the other component must bound a disk in $R_-$ because $R_\pm$ are incompressible surfaces by tautness of $M$, it follows that $C$ along with these two disks cobound a 3-ball by irreducibility of $M$, which would violate the definition of a product annulus. We can isotope the suture of $M'$ as in Figure 1 so that 
        \[\aligned
            &M'=M\setminus (C\times (-1,1)),
        \\&
            R_\pm'=R_\pm\setminus (\partial C\times (-1,1)),\quad \gamma'=\gamma\cup (C\times\{\pm1\}).\endaligned
        \] Let $\alpha$ be the boundary circle of $C$ on $R_-$, we can view the pair $(M,R_-)$ as
        \[(M,R_-)=(M',R_-')\cup (C\times [-1,1], \alpha\times [-1,1]).\]
        The intersection of the two pairs are two copies of $(C,\alpha)$, which have relative $L^2$-torsion equal to 1. So we have
        \[
            \tautwo(M,R_-)=\tautwo(M',R_-')\cdot\tautwo(C\times [-1,1], \alpha\times [-1,1])=\tautwo(M',R_-')
        \]
        by Lemma \ref{BasicOfRelativeTorsion}.

        When $C$ is a product disk, we isotope the suture of $M'$ as in Figure 1 so that
        \[\aligned
            &M'=M\setminus (C\times(-1,1)),\\&
            R_\pm'=R_\pm\setminus (\partial C\times (-1,1)),\quad
            \gamma'=(\gamma\cap M')\cup (C\times\{\pm1\}).\endaligned
        \]
        The intersection $R_-\cap \partial C$ is an arc which we denote by $\beta$, we can view the pair $(M,R_-)$ as 
        \[(M,R_-)=(M',R_-')\cup (C\times [-1,1],\beta\times[-1,1]).\] The intersection of the two pairs are two copies of $(C,\beta)$, which have relative $L^2$-torsion equal to 1. So we have
        \[
            \tautwo(M,R_-)=\tautwo(M',R_-')\cdot\tautwo(C\times [-1,1], \beta\times [-1,1])=\tautwo(M',R_-')
        \]
        again by Lemma \ref{BasicOfRelativeTorsion}. This finishes the proof.
    \end{proof}
\begin{figure}        
\label{FigureOfProductAnnuliAndDisks}

        \centering
        \begin{subfigure}
        
        \def\svgwidth{1\columnwidth}
        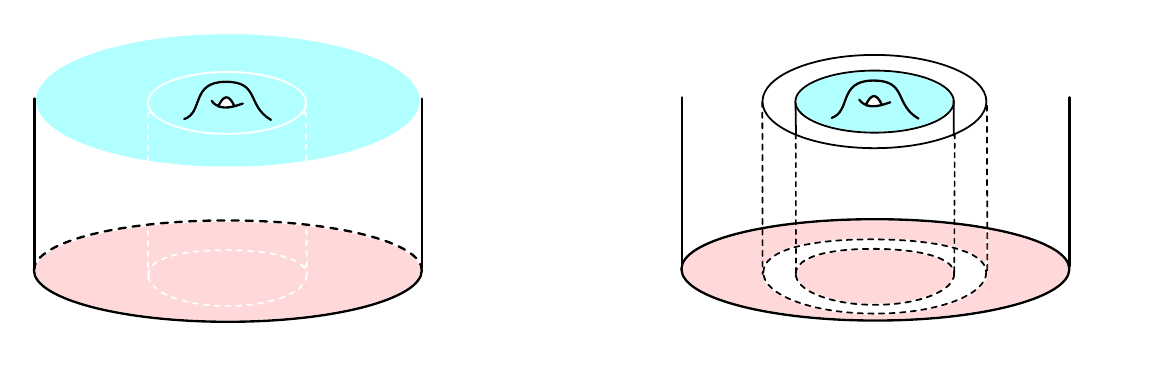
        \end{subfigure}
        \begin{subfigure}
        
            \def\svgwidth{1\columnwidth}
             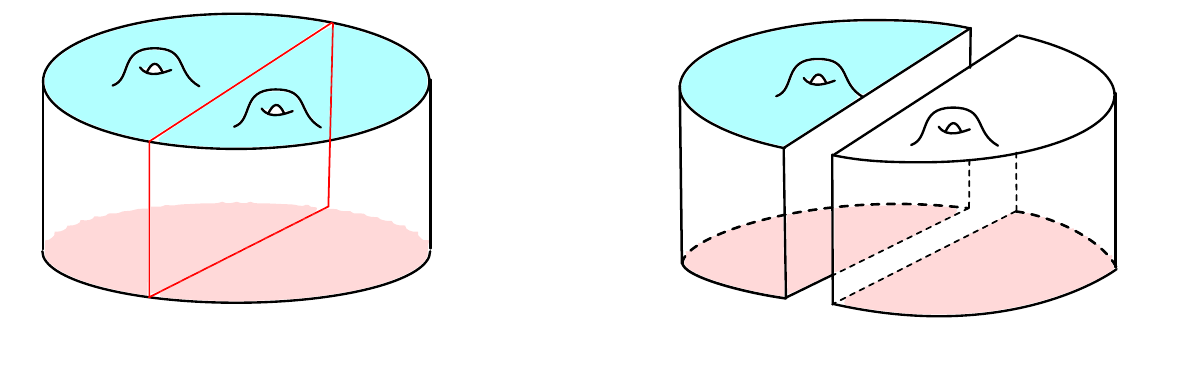
        \end{subfigure}
        \caption{Decomposition along a product annulus (the above part) and a product disk (the lower part). By definition $R_+'$ and $R_-'$ each contains a copy of $C$ respectively, but one can isotope (by absorbing the copies of $C$ into the suture) so that $R_\pm'=R_\pm\setminus(\partial C\times (-1,1))$ as in the figure on the right. The two copies of $C$ in $\partial M'$ then become subsets of $\gamma'$.}
     \end{figure}

    \begin{thm2}
 Let $N$ be a connected, orientable, irreducible, compact 3-manifold with empty or toral boundary. Then for any primitive $\phi\in H^1(N;\Z)$, we have
        \[
            C(N,\phi)=\tautwo(\Gamma(\phi), R_-\Gamma(\phi)).
        \]
    \end{thm2}
    \begin{proof}
    Let $F=F(\phi)$ be any facet surface for $\phi$, then $F$ is dual to $n\phi$ for some $n\not=0$. We have 
        \[
            C(M,\phi)=C(M,n\phi)=\tautwo(M\backslash\backslash F,F_-)
        \]
    by Theorem \ref{LeadingCoefficientEqualsRelativeTorsion} and dilatation invariance. The guts $\Gamma(\phi)$ (which may depend on the facet surface when the Thurston norm degenerates) is obtained from $M\backslash\backslash F$ by decomposing along a disjoint union of product annuli and product disks, and deleting product sutured manifold components. By the previous lemma and Lemma \ref{BasicOfRelativeTorsion} that product sutured manifold has trivial relative $L^2$-torsion, we have 
        \[
            \tautwo(M\backslash\backslash F,F_-)=\tautwo(\Gamma(\phi), R_-\Gamma(\phi)).
        \]
        This finishes the proof.
    \end{proof}

    \subsection{The leading coefficient of an open Thurston cone}
     
    The remaining part of the paper is devoted to the proof of Theorem \ref{InvarianceOfLeadingCoefficientOpenCone}. The following lemma is from \cite[Corollary 5.8]{duan2022positivity} generalizing Theorem \ref{LiuConvexityOfFugledeKadison}. In fact, Theorem \ref{LiuConvexityOfFugledeKadison} is a special case of this lemma when $n=1$.

    \begin{lem}\label{MultiAlexanderConvexity}
        Let $G$ be a finitely generated, residually finite group. Suppose that $A$ is a square matrix over $\C G$ with $\rdet G(A)>0$. Then for any positive real number $t_0$, any positive integer $n$ and any cohomology classes $\phi_1,\ldots,\phi_n\in H^1(G;\R)$, the following function
        \[
            W:\R^n\ra \R,\quad W(s_1,\ldots,s_n)=\log\rdet G\Big(\kappa(s_1\phi_1+\cdots+s_n\phi_n,t_0)A\Big)
        \]
        is convex.
    \end{lem}

    The following lemma allows us to  write down a simple formula for the $L^2$-Alexander torsion of hyperbolic 3-manifolds. 
    
     \begin{lem}\label{L2AlexanderTorsionFormulaForHyperbolic3Manifold}
        Let $N$ be a connected, orientable, compact 3-manifold whose interior admits a hyperbolic structure and let $\mathcal C$ be an open Thurston cone of $H^1(N;\R)$. Denote by $G$ the fundamental group of $N$, then there exists a positive integer $d$, a square matrix $A$ over $\Z G$ and an element $h\in G$ such that 
        \[
            (\tautwo(N,\phi)(t))^d \doteq \rdet{G} (\kappa(\phi,t)A)/\max\{1,t^{\phi(h)}\}^2
        \]
        for all $\phi\in H^1(N;\R)$ and $t>0$. Moreover, we can require that $\phi(h)\geqslant 0$ for all $\phi\in\mathcal C$.
    \end{lem}
    \begin{proof}
        By Agol's RFRS criterion for virtual fibering \cite{agol2008criteria} and the virtual specialness of hyperbolic 3-manifolds \cite{agol2013virtual}, there is a finite covering $p_1: N'\ra N$ such that $ N'$ fibers over the circle and that $p_1^*\mathcal C$ lies in the closure of a fibered cone $\mathcal C'$ of $H^1( N';\R)$. Pick a primitive class $\psi\in \mathcal C'$ and consider the fibration determined by $\psi$ with fiber surface $\Sigma$, then $\Sigma$ is connected and $N'$ is homeomorphic to a mapping torus of a surface automorphism $f:\Sigma\ra \Sigma$, namely
        \[
            N'=(\Sigma\times [-1,1])/(x,-1)\sim(f(x),1).
        \]
        Moreover, since $N'$ is hyperbolic, then $\chi(\Sigma)<0$ and $f$ is isotopic to a pseudo-Anosov surface automorphism, which preserves two mutually transverse invariant singular foliations on $\Sigma$ (see \cite{FLP2012ThurstonWorkOnSurface}). So up to a homeomorphism of $N'$ we can assume that $f$ is pseudo-Anosov. Since $f$ preserves the non-empty finite set of singular points of the foliations, there is a positive integer $n$ such that $f^n:\Sigma\ra \Sigma$ fixes a singular point $P\in \Sigma$ of the foliations. Let $p_2:N''\ra N'$ be the $n$-fold covering of $N'$ where
        \[
            N''=(\Sigma\times [-1,1])/(x,-1)\sim(f^n(x),1).
        \] The fibration of $N''$ coming from this mapping torus structure is exactly the fibration of the fibered class $p_2^*\psi$. Denote by $\mathcal C''\subset H^1(N'';\R)$ the fibered cone containing $p_2^*\psi$, then  $C''$ contains $p_2^*C'$. We denote by $G'':= \pi_1(N'',P\times\{-1\})$ the fundamental group of $N''$ with basepoint $P\times\{-1\}$,
        and let $h$ be the element represented by the oriented loop $\{P\}\times [-1,1]$. This loop is a closed trajectory of the suspension flow of $N''$, then $\phi(h)>0$ for all classes $\phi\in \mathcal C''$ by work of Fried (see \cite[Theorems 14.10--14.11]{FLP2012ThurstonWorkOnSurface}). By continuity, it follows that $\phi(h)\geqslant 0$ for all $\phi$ in the closure of $\mathcal C''$.
        
        Consider the composite covering
        \[
            p:=p_1\circ p_2:N''\ra N.
        \]
        Since $\mathcal C$ lifts to a subset of the closure of $\mathcal C'$ via $p_1$, and $\mathcal C'$ lifts to a subset of the fibered cone $\mathcal C''$ via $p_2$, it follows that $\mathcal C$ lifts to a subset of the closure of $\mathcal C''$ via $p:N''\ra N$. In particular we have $p^*\phi(h)\geqslant0$ for all $\phi\in\mathcal C$, or equivalently,
        \begin{equation}\notag
            \phi(p_*h)\geqslant 0 \text{\quad for all\quad} \phi\in \mathcal C
        \end{equation}
        where $p_*:\pi_1(N'')\ra \pi_1(N)$ is the induced injection on fundamental groups.

        In order to compute the $L^2$-Alexander torsion of $N''$, we construct an explicit CW-complex homotopy equivalent to $N''$. The following procedure is essentially the same as \cite[Section 8.3]{dubois2015l2} and \cite[proof of Theorem 4.5]{duan2022positivity}. Construct a CW complex $X$ modeled on $\Sigma$ with a single 0-cell $P$, $k$ 1-cells, and a 2-cell. By CW approximation, there is a cellular map $g:X\ra X$ homotopic to $f^n$ with $g(P)=P$. Then the mapping torus $X_g$ is homotopy equivalent to $N''$, which is a simple homotopy equivalence since the Whitehead group of a fibered 3-manifold is trivial, see \cite[Theorem 19.4, Theorem 19.5]{waldhausen1978algebraicPart2}. Hence by simple homotopy invariance of $L^2$-Alexander torsions \cite[Theorem 3.93]{luck2002l2} we have
\[\tautwo(N,\phi)(t) \doteq  \tautwo(X_g,\phi)(t).\] After choosing a suitable basis, the CW-chain complex of the universal covering $\widehat {X_g}$ as a free $\Z G''$-module has the form
    \[C_*(\widehat{X_g})=(0\longrightarrow \Z G''\stackrel{\partial_3}\longrightarrow \Z G''^{k+1} \stackrel{\partial_2}\longrightarrow \Z G''^{k+1} \stackrel{\partial_1} \longrightarrow \Z G''\longrightarrow 0 )\]
where
\[\partial_3=(1-h,*),\quad \partial_2=\begin{pmatrix}
* & *\\
A & *
\end{pmatrix},\quad\partial_1=\begin{pmatrix}
*\\
1-h
\end{pmatrix}
\]
in which $``*"$ stands for matrices of appropriate size and $A$ is a matrix over $\Z G''$ of size $k\times k$ with $\rdet {G''}A>0$. Hence for all $\phi\in H^1(N'';\R)$,
\[\aligned
            \tautwo(N'',\phi)(t)&\doteq\tautwo(X_g,\phi)(t)\\&\doteq\rdet{G''}(\kappa(\phi,t)A)\cdot \rdet {G''}(\kappa(\phi,t)(1-h))^{-2}\\
            &=\rdet {G''}(\kappa(\phi,t)A)/\max\{1,t^{\phi(h)}\}^2 \endaligned
        \]
        Here we used the Turaev's method of calculating torsions (see \cite[Lemma 3.2]{dubois2015l2}) and the Fuglede--Kadison determinant for free abelian groups (see \cite[Lemma 2.8]{dubois2015l2}).
  Let $d$ be the covering degree of $p:N''\ra N$, then
\[\aligned
    (\tautwo(N,\phi)(t))^d&\doteq\tautwo(N'',p^*\phi)(t)\\
    &\doteq\rdet {G''}(\kappa(p^*\phi,t)A)/\max\{1,t^{p^*\phi(h)}\}^2\\&=\rdet {G}(\kappa(\phi,t)p_*A)/\max\{1,t^{\phi(p_*h)}\}^2  \endaligned
\]
where $p_*A$ is a square matrix over $\Z G$ and $p_*h\in G$; the first equality follows from Theorem \ref{KnownPropertiesOfL2Alexander}\,(2) and the third equality follows from Lemma \ref{BasicsOfRegularFugledeKadison}\,(4). This finishes the proof of the lemma.
    \end{proof}
    
The following lemma about convex functions will be used in the proof of Theorem \ref{InvarianceOfLeadingCoefficientOpenCone}.

\begin{lem}\label{LemmaOfConvexFunctions}
            Let $f:\R^2\ra \R$ be a real function, such that the following conditions hold:
            \begin{enumerate}[{\rm(i)}]
                \item The restrictions  $f|_{\Omega_1}$ and $f|_{\Omega_2}$ are convex, where $$\Omega_1=\{(x,y)\in \R^2\mid x,y\geqslant0\},\quad\Omega_2=\{(x,y)\in \R^2\mid x,y\leqslant0\}.$$
                \item For any $\lambda\in [0,1]$, there exists $b_1(\lambda),b_2(\lambda)\in \R$ and $C(\lambda)>0$ such that
                \[
                    \lim_{s\ra-\infty}( f(\lambda s,(1-\lambda)s)-b_1(\lambda)s)=\lim_{s\ra+\infty}( f(\lambda s,(1-\lambda)s)-b_2(\lambda)s)=C(\lambda)
                \]
                and $b_2(\lambda)-b_1(\lambda)=1$.
            \end{enumerate}
            Then $C(\lambda)\leqslant \lambda C(1)+(1-\lambda)C(0)$ for any $\lambda\in[0,1]$.
        \end{lem}
        The condition (ii) of this lemma can be interpreted as follows: the graph of $f$ restricted to the real line $L_\lambda:=\{(\lambda s,(1-\lambda)s)\mid s\in\R\}$ is asymptotically two rays when $s\ra\pm\infty$ with the difference of the slopes equal to $1$. Moreover, the two asymptotic rays share the same additive constant $C(\lambda)$.

        \begin{proof}[{Proof of Lemma \ref{LemmaOfConvexFunctions}}]
            Subtracting a linear function $(x,y)\mapsto b_1(1)x+b_1(0)y$ from $f(x,y)$ does not change the conditions (i), (ii) and the function $C(\lambda)$, so we can assume without loss of generality that $b_1(0)=b_1(1)=0$ and $b_2(0)=b_2(1)=1$. Equivalently, 
            \[\aligned
                &\lim_{s\ra-\infty} f(s,0)=\lim_{s\ra+\infty} (f(s,0)-s)=C(1),\\&\lim_{s\ra-\infty} f(0,s)=\lim_{s\ra+\infty} (f(0,s)-s)=C(0).\endaligned
            \]
            By condition (i), for fixed $\lambda\in[0,1]$ and $s\in\R$ we have
            \[
                LHS(s):=f(\lambda s,(1-\lambda)s)\leqslant \lambda f(s,0)+(1-\lambda)f(0,s)=:RHS(s).
            \]
            Let $s\ra -\infty$, by condition (ii) we have
            \[
                LHS(s)=b_1(\lambda)s+C(\lambda)+o(1),\quad RHS(s)=\lambda C(1)+(1-\lambda)C(0)+o(1),
            \] so we must have $b_1(\lambda)\geqslant 0$. Let $s\ra +\infty$, by condition (ii) we have
            \[
                LHS(s)=b_2(\lambda)s+C(\lambda)+o(1),\quad RHS(s)=s+\lambda C(1)+(1-\lambda)C(0)+o(1),
            \]so we must have $b_2(\lambda)\leqslant 1$. But $b_2(\lambda)-b_1(\lambda)=1$ by condition (ii), so $b_1(\lambda)=0$ and $b_2(\lambda)=1$. Hence
            \[
                C(\lambda)\leqslant \lambda C(1)+(1-\lambda)C(0).
            \]
            The proof is finished since the choice of $\lambda\in[0,1]$ is arbitrary.
        \end{proof}

    \begin{thm3}
        Let $N$ be a connected, orientable, irreducible, compact 3-manifold with empty or toral boundary. Then for each open Thurston cone $\mathcal C$ of $H^1(N;\R)$, the leading coefficient $C(N,\phi)$ is constant for all classes $\phi\in \mathcal C$.
    \end{thm3}
    
    \begin{proof}

\textit{Step 1.} Reduce to the case where $N$ is hyperbolic and thus has non-degenerate Thurston norm.
    
        Let $N_1,\ldots,N_r$ be the components in the JSJ-decomposition of $N$, each component is either hyperbolic or Seifert fibered. By Theorem \ref{PropertiesOfLeadingCoefficient} we have
        \[
        C(N,\phi)=\prod_{i=1}^r C(N_i,\phi|_{N_i}).
        \]
        We claim that if two classes $\phi_1,\phi_2\in H^1(N;\R)$ are in the same open Thurston cone, then the restrictions $\phi_1|_{N_i}$ and $\phi_2|_{N_i}$ are in the same open Thurston cone of $H^1(N_i;\R)$. If this claim is true, then without loss of generality, we can assume that $N$ is either Seifert fibered or hyperbolic. Every cohomology class of a Seifert fibered manifold has leading coefficient 1 by Theorem \ref{PropertiesOfLeadingCoefficient}, we only need to consider the case where $N$ is hyperbolic. In particular, $N$ has non-degenerate Thurston norm.

        It remains to prove the claim. Note that two classes $\phi_1,\phi_2$ are in the same open Thurston cone of $H^1(N;\R)$ if and only if there is an open segment $(u,v)$ containing $\phi_1,\phi_2$ on which the Thurston norm $x_N$ is linear. It follows from \cite[Proposition 3.5]{eisenbud1985three} that 
        \[
            x_N(\phi)=\sum_{i=1}^r x_{N_i}(\phi|_{N_i})
        \]
        where $\phi\in H^1(N;\R)$ is any class and $\phi|_{N_i}$ is the restriction of $\phi$ to $N_i$ (alternatively, one can use Theorem \ref{KnownPropertiesOfL2Alexander}\,(3),\,(7) to see this fact). Then by sub-additivity of the Thurston norm, $x_N$ being linear on the segment $(u,v)$ will imply that $x_{N_i}$ is linear on the segment $(u|_{N_i},v|_{N_i})$ for all $i=1,\ldots,r$. If $(u|_{N_i},v|_{N_i})$ is an open segment of $H^1(N_i;\R)$, then $\phi_1|_{N_i},\phi_2|_{N_i}$ are in the same open Thurston cone of $H^1(N_i;\R)$. The other possibility is that $(u|_{N_i},v|_{N_i})$ degenerates to a point, in this case $\phi_1|_{N_i}=\phi_2|_{N_i}$, in particular $\phi_1|_{N_i},\phi_2|_{N_i}$ are in the same open Thurston cone of $H^1(N_i;\R)$. This finishes the proof of the claim.

        \textit{Step 2.} We prove that for any open Thurston cone $\mathcal C$, the function
        \[
            \phi\longmapsto \log C(N,\phi)
        \]
        is convex over the open Thurston face $\{\phi\in \mathcal C\mid x_N(\phi)=1\}$. Precisely, we prove that for any $\phi_1,\phi_2\in \mathcal C$ with $x_N(\phi_1)=x_N(\phi_2)=1$, and any $\lambda\in[0,1]$,
        \begin{equation}\label{LinearityOfLeadingCoefficient}\tag{$\dagger$}
            \log C(N,\lambda\phi_1+(1-\lambda)\phi_2)\leqslant\lambda \log C(N,\phi_1)+(1-\lambda)\log C(N,\phi_2).
        \end{equation}

        When dealing with convexity properties of the $L^2$-Alexander torsions, it is often convenient to apply the change of variable $t=:e^s,\ s\in \R$. By Lemma \ref{BasicOfRelativeTorsion}\,(1), 
        \[\tautwo(N,\phi)(t)=\tautwo(N,\phi)(e^s)=\tautwo(N,s\phi)(e).\] 
        By Theorem \ref{KnownPropertiesOfL2Alexander}\,(7), the $L^2$-Alexander torsion is asymptotically monomial in both ends, i.e. there exists $b_1, b_2\in\R$ depending on $\phi$ such that
        \begin{equation}\tag{*}\label{*}
            \lim_{s\ra -\infty}(\log\tautwo(N,s\phi)(e)-b_1s)=\lim_{s\ra +\infty}(\log\tautwo(N,s\phi)(e)-b_2s)=\log C(N,\phi)
       \end{equation} 
       where $b_2-b_1=x_N(\phi)$.
       
       We fix two arbitrary classes $\phi_1,\phi_2\in \mathcal C$ with $x_N(\phi_1)=x_N(\phi_2)=1$. Define
        \[
            f:\R^2\longrightarrow \R,\quad f(s_1,s_2):=\log\tautwo(N,s_1\phi_1+s_2\phi_2)(e).
        \]
        We wish to apply Lemma \ref{LemmaOfConvexFunctions} to $f$. 
        By Lemma \ref{L2AlexanderTorsionFormulaForHyperbolic3Manifold}, there exists a positive integer $d$, a square matrix $A$ over $\Z G$ and $h\in G$, such that $\phi(h)\geqslant 0$ for all $\phi\in\mathcal C$ and
        \[
            (\tautwo(M,\phi)(t))^d=\rdet{G} (\kappa(\phi,t)A)/\max\{1,t^{\phi(h)}\}^2.
        \]
        Hence
        \[\aligned
            f(s_1,s_2)&=\frac{\log \rdet G(\kappa(s_1\phi_1+s_2\phi_2,e)A)}{d}-\frac{2\max\{0,s_1\phi_1(h)+s_2\phi_2(h)\}}{d}\\&=:I_1-I_2.\endaligned
        \]
        The first term $I_1$ is convex over $\R^2$ by Lemma \ref{MultiAlexanderConvexity}. The second term $I_2$ is linear when restricted to $\{(s_1,s_2)\in \R^2\mid s_1,s_2\geqslant0\}$ or $\{(s_1,s_2)\in \R^2\mid s_1,s_2\leqslant0\}$ (here we used $\phi_1(h)\geqslant0,\ \phi_2(h)\geqslant0$). So condition (i) of Lemma \ref{LemmaOfConvexFunctions} holds for $f$.
        As for condition (ii), by definition for any $s\in \R$ and $\lambda\in[0,1]$,
        \[
            f(\lambda s,(1-\lambda)s)=\log\tautwo(N,s(\lambda\phi_1+(1-\lambda)\phi_2))(e).
        \]
        Put $\phi=\lambda\phi_1+(1-\lambda)\phi_2$ in \eqref{*}, it follows that there exists $b_1,b_2\in \R$ such that
                \[
                    \lim_{s\ra-\infty}( f(\lambda s,(1-\lambda)s)-b_1s)=\lim_{s\ra+\infty}( f(\lambda s,(1-\lambda)s)-b_2s)=\log C(N,\lambda\phi_1+(1-\lambda)\phi_2)
                \]
                where $b_2-b_1=x_N(\lambda\phi_1+(1-\lambda)\phi_2)=1$ since $\phi_1,\phi_2$ belong to the same open Thurston face. So condition (ii) holds for $f$. Now Lemma \ref{LemmaOfConvexFunctions} implies that
                \[
                    \log C(N,\lambda\phi_1+(1-\lambda)\phi_2)\leqslant\lambda\log C(N,\phi_1)+(1-\lambda)\log C(N,\phi_2)
                \]
        for any $\lambda\in[0,1]$. This proves \eqref{LinearityOfLeadingCoefficient}.
        
\textit{Step 3.} We have just proved that the function $\phi\mapsto\log C(N,\phi)$ is convex over the open Thurston face of $\mathcal C$. This function is upper semi-continuous by Theorem \ref{PropertiesOfLeadingCoefficient}\,(6), so it is continuous over the open Thurston face. By dilatation invariance (Theorem \ref{PropertiesOfLeadingCoefficient}\,(1)) we know that the function
\[
    \phi\longmapsto C(N,\phi),\quad \phi\in\mathcal C
\]
is continuous over the whole open Thurston cone $\mathcal C$.

By Corollary \ref{FinitelyManyGutsInOpenThurstonCone} there are at most finitely many equivalent classes of guts $\Gamma(z)$ for primitive $z\in\mathcal C$. By Theorem \ref{LeadingCoefficientEqualsGuts} and dilatation invariance of the leading coefficient, there are only finitely many possible values $C(N,\phi)$ for rational classes $\phi\in\mathcal C$. It follows that the leading coefficient must be constant over $\mathcal C$ by continuity. This finishes the proof of the Theorem.
    \end{proof}

    Given two nonzero cohomology classes $\phi,\psi\in H^1(N;\R)$, we say $\phi$ \textit{is subordinate to} $\psi$ if $\phi$ belongs to the closure of the unique open Thurston cone containing $\psi$. Here it is convenient to assume that $0\in H^1(N;\R)$ is subordinate to all other classes. We end the paper with a corollary which follows from Theorem \ref{InvarianceOfLeadingCoefficientOpenCone} and the upper semi-continuity of the leading coefficient function (Theorem \ref{PropertiesOfLeadingCoefficient} (6)).
  
  \begin{cor}\label{Subordination}
      Let $N$ be a connected, orientable, irreducible, compact 3-manifold with empty or toral boundary, then there are only finitely many possible values for the leading coefficients $C(N,\phi)$ as $\phi$ varies in $H^1(N;\R)$. Given $\phi,\psi\in H^1(N;\R)$ with $\phi$ subordinate to $\psi$, then $C(N,\phi)\geqslant C(N,\psi)$.
  \end{cor}
  
  \bibliography{ref.bib}

\newcommand{\etalchar}[1]{$^{#1}$}
\providecommand{\bysame}{\leavevmode\hbox to3em{\hrulefill}\thinspace}
\providecommand{\MR}{\relax\ifhmode\unskip\space\fi MR }
\providecommand{\MRhref}[2]{%
  \href{http://www.ams.org/mathscinet-getitem?mr=#1}{#2}
}
\providecommand{\href}[2]{#2}
\begin{thebibliography}{BAFH22}

\bibitem[AFW15]{aschenbrenner20153}
M.~Aschenbrenner, S.~Friedl, and H.~Wilton, \emph{{3-manifold Groups}}, EMS series of lectures in mathematics, European Mathematical Society, 2015.

\bibitem[AGM13]{agol2013virtual}
Ian Agol, Daniel Groves, and Jason Manning, \emph{{The virtual Haken conjecture}}, Doc. Math \textbf{18} (2013), no.~1, 1045--1087.

\bibitem[Ago08]{agol2008criteria}
Ian Agol, \emph{{Criteria for virtual fibering}}, Journal of Topology \textbf{1} (2008), no.~2, 269--284.

\bibitem[AZ22]{agol2022guts}
Ian Agol and Yue Zhang, \emph{{Guts in sutured decompositions and the Thurston norm}}, arXiv preprint arXiv:2203.12095 (2022).

\bibitem[BAFH22]{ben2022leading}
Fathi Ben~Aribi, Stefan Friedl, and Gerrit Herrmann, \emph{{The leading coefficient of the $L^2$-Alexander torsion}}, Annales de l'Institut Fourier \textbf{72} (2022), no.~5, 1993--2035.

\bibitem[BRW14]{Boileau2014Finitenessof3-manifolds}
Michel Boileau, J.~Hyam Rubinstein, and Shicheng Wang, \emph{Finiteness of 3-manifolds associated with non-zero degree mappings}, Comment. Math. Helv. \textbf{89} (2014), no.~1, 33--68. \MR{3177908}

\bibitem[CFM97]{mathai1997determinant}
A.~Carey, M.~Farber, and V.~Mathai, \emph{{Determinant lines, von Neumann algebras and $L^2$-torsion}}, Journal für die reine und angewandte Mathematik \textbf{1997} (1997), no.~484, 153--182.

\bibitem[DFL16]{dubois2015l2}
J{\'e}r{\^o}me Dubois, Stefan Friedl, and Wolfgang L{\"u}ck, \emph{{The $L^2$-Alexander torsions of 3-manifolds}}, Journal of Topology \textbf{9} (2016), no.~3, 889--926.

\bibitem[Dua24]{duan2022positivity}
Jianru Duan, \emph{{On the positivity of twisted $L^2$-torsion for 3-manifolds}}, Algebraic \& Geometric Topology \textbf{24} (2024), no.~4, 2307–2329.

\bibitem[ENN85]{eisenbud1985three}
David Eisenbud, Walter Neumann, and Walter~D Neumann, \emph{{Three-dimensional link theory and invariants of plane curve singularities}}, no. 110, Princeton University Press, 1985.

\bibitem[FF98]{Freedman1998KneserHakenFiniteness}
Benedict Freedman and Michael~H. Freedman, \emph{Kneser-{H}aken finiteness for bounded {$3$}-manifolds locally free groups, and cyclic covers}, Topology \textbf{37} (1998), no.~1, 133--147. \MR{1480882}

\bibitem[FL19]{friedl2019l2}
Stefan Friedl and Wolfgang L{\"u}ck, \emph{{The $L^2$-torsion function and the Thurston norm of 3-manifolds}}, Comment. Math. Helv \textbf{94} (2019), no.~1, 21--52.

\bibitem[FLP{\etalchar{+}}12]{FLP2012ThurstonWorkOnSurface}
Albert Fathi, François Laudenbach, Valentin Poénaru, Djun~M. Kim, and Dan Margalit, \emph{Thurston's work on surfaces (mn-48)}, vol.~48, Princeton University Press, 2012.

\bibitem[Gab87]{gabai1987foliations}
David Gabai, \emph{{Foliations and the topology of 3-manifolds. II}}, Journal of Differential Geometry \textbf{26} (1987), no.~3, 461--478.

\bibitem[Hem87]{hempel1987residual}
John Hempel, \emph{{Residual finiteness for 3-manifolds}}, Combinatorial group theory and topology (Alta, Utah, 1984) \textbf{111} (1987), 379--396.

\bibitem[Her23]{herrmann2023sutured}
Gerrit Herrmann, \emph{{Sutured manifolds and $l^2$-Betti numbers}}, The Quarterly Journal of Mathematics \textbf{74} (2023), no.~4, 1435--1455.

\bibitem[Kam19]{kammeyer2019introduction}
Holger Kammeyer, \emph{{Introduction to $L^2$-invariants}}, vol. 2247, Springer Nature, 2019.

\bibitem[Kie20]{Kielak2019RFRSGroupsandVirtualFibering}
Dawid Kielak, \emph{Residually finite rationally solvable groups and virtual fibring}, J. Amer. Math. Soc. \textbf{33} (2020), no.~2, 451--486. \MR{4073866}

\bibitem[Liu17]{liu2017degree}
Yi~Liu, \emph{{Degree of $L^2$-Alexander torsion for 3-manifolds}}, Inventiones mathematicae \textbf{207} (2017), no.~3, 981--1030.

\bibitem[LS99]{luck19992}
Wolfgang L{\"u}ck and Thomas Schick, \emph{{$L^2$-torsion of hyperbolic manifolds of finite volume}}, Geometric \& Functional Analysis GAFA \textbf{9} (1999), no.~3, 518--567.

\bibitem[L{\"u}c02]{luck2002l2}
Wolfgang L{\"u}ck, \emph{{$L^2$-invariants: theory and applications to geometry and K-theory}}, vol.~44, Springer, 2002.

\bibitem[L{\"u}c18]{luck2018twisting}
\bysame, \emph{{Twisting $L^2$-invariants with finite-dimensional representations}}, Journal of Topology and Analysis \textbf{10} (2018), no.~04, 723--816.

\bibitem[LZ06a]{li2006alexander}
Weiping Li and Weiping Zhang, \emph{{An $L^2$-Alexander-Conway Invariant for Knots and the Volume conjecture}}, vol.~10, pp.~303--312, World Scientific Hackensack, NJ, 2006.

\bibitem[LZ06b]{li2006l2}
\bysame, \emph{{An $L^2$-Alexander invariant for knots}}, Communications in Contemporary Mathematics \textbf{8} (2006), no.~02, 167--187.

\bibitem[Mun66]{Munkres1961ElementaryDiff}
James~R. Munkres, \emph{Elementary differential topology}, revised ed., Annals of Mathematics Studies, vol. No. 54, Princeton University Press, Princeton, NJ, 1966, Lectures given at Massachusetts Institute of Technology, Fall, 1961. \MR{198479}

\bibitem[Thu86]{thurston1986norm}
William~P Thurston, \emph{{A norm for the homology of 3-manifolds}}, Mem. Amer. Math. Soc. \textbf{59} (1986), 99--130.

\bibitem[Ver84]{Verona1984StratifiedMappings}
Andrei Verona, \emph{Stratified mappings---structure and triangulability}, Lecture Notes in Mathematics, vol. 1102, Springer-Verlag, Berlin, 1984. \MR{771120}

\bibitem[Wal78]{waldhausen1978algebraicPart2}
Friedhelm Waldhausen, \emph{{Algebraic k-theory of generalized free products, part 2}}, Annals of Mathematics \textbf{108} (1978), no.~2, 205--256.

\bibitem[Whi40]{Whitehead1940OnC1Complexes}
J.~H.~C. Whitehead, \emph{On {$C^1$}-complexes}, Ann. of Math. (2) \textbf{41} (1940), 809--824. \MR{2545}

\end{thebibliography}

\end{document}